\documentclass[reqno]{amsart}

\usepackage{amssymb}

\usepackage{euscript}
\newtheorem{Lemma}{Lemma}[section]
\newtheorem{Theorem}[Lemma]{Theorem}

\newcommand{\graph}{\operatorname{graph}}
\newcommand{\de}[2]{\frac{\partial #1}{\partial #2}}

\begin{document}
\title{Hopf type rigidity for Thermostats}

\author[Y. M. Assylbekov]{Yernat M. Assylbekov}
\address{Department of Mathematics, University of Washington, Seattle, WA 98195-4350, USA}
\email{y\_assylbekov@yahoo.com}

\author[N.S. Dairbekov]{Nurlan S. Dairbekov}
\address{Kazakh British Technical University,
Tole bi 59, 050000 Almaty, Kazakhstan }
\email{Nurlan.Dairbekov@gmail.com}

\begin{abstract}
We study the motion of a particle on a Riemannian 2-torus under the influence of a magnetic field and a Gaussian thermostat. We prove a Hopf type rigidity for this dynamical system without conjugate points.
\end{abstract}

\maketitle

\section{Introduction and statement of the result}
It was proved by E. Hopf \cite{H} that a Riemannian 2-torus without conjugate
points is necessarily flat. To higher
dimensions this result was generalized in \cite{BI}.
The results of \cite{B,BP,K} show that this type of rigidity
holds also for dynamical systems more general than the geodesic flow.
In this paper we establish a Hopf type rigidity for a thermostat on a 2-torus.

Let $(M,g)$ be a closed oriented Riemannian surface, and $SM$ its unit sphere bundle
with canonical projection $\pi:SM\to M$, $\pi(x,v)=x$.
Given a function $f\in C^\infty(M)$ and a smooth vector field ${\bf e}$ on $M$,
let $\lambda\in C^\infty(SM)$ be the function on $SM$ given by
\begin{equation}\label{lambda}
\lambda(x,v):=f(x)+\langle {\bf e}(x), iv\rangle,
\end{equation}
where $i$ indicates the rotation by $\pi/2$ according to the orientation of $M$.
A curve $\gamma(t)$ on $M$ satisfying the equation
\begin{equation}\label{thermo}
D_t\dot\gamma=\lambda(\gamma,\dot\gamma)i\dot\gamma
\end{equation}
is called a {\it thermostat geodesic}. Here and futher $D_t$ denotes the covariant derivative along $\gamma$.
Equation \eqref{thermo} also
defines a flow $\phi$ on $SM$, to be called the flow
of the {\it thermostat} $(M,g,f,{\bf e})$.

The flow $\phi$ reduces to the geodesic flow when ${\bf e} = f = 0$. If ${\bf  e} = 0$, then $\phi$ is the magnetic flow associated with the magnetic field $\Omega=f\Omega_a$, where $\Omega_a$ is the area form of $M$. If $f = 0$, we obtain the {\it Gaussian thermostat},  which is reversible in the sense that the flip $(x, v)\mapsto(x,-v)$ conjugates $\phi_t$ with $\phi_{-t}$ (just as in the case of geodesic flows). Thus the dynamical system governed by (\ref{thermo}) describes the motion of a particle on $(M,g)$ under the combined influence of a magnetic field $f\Omega_a$ and a thermostat with external field ${\bf e}$.

Magnetic flows were firstly considered in \cite{AS, Ar} and it was shown in \cite{ArG, Ko, N1, N2, NS, PP} that they are related to dynamical systems, symplectic geometry, classical mechanics and mathematical physics. Gaussian thermostats provide interesting models in non-equilibrium statistical mechanics \cite{G,GR,R}.

We define the exponential map at $x\in M$ to be
$$
\exp_x^{\lambda}(tv)=\pi(\phi_t(x,v)), \quad t>0, \quad v\in S_xM.
$$

We say that the thermostat in question {\it has no conjugate points} if $\exp_x^{\lambda}$ is a local diffeomorphism for all $x\in M$. The main result of this paper is a Hopf type rigidity for a thermostat on a 2-torus $\mathbb T^2$.

\begin{Theorem}\label{theorem A} A thermostat $({\mathbb T}^2,g,f,{\bf e})$ has no conjugate points if and only if $f=0$ and there is $U\in C^\infty({\mathbb T}^2)$ such that $\mbox{\rm div}({\bf e}+\nabla U)=0$ and the metric $g_1=e^{-2U}g$ is flat.
\end{Theorem}

The proof of the main theorem follows the original scheme by E.~Hopf with modifications
for our dynamical system.
Section 2 gives an explanation of how the function $U$ in
Theorem \ref{theorem A} is shosen. In Section 3 we collect some preliminary information on
thermostats. In Section 4 we restate the no-conjugate-points condition
in terms of the Jacobi equation in analogy with the case of a geodesic flow.
One of the main ingredients in the proof of Theorem \ref{theorem A} is
the construction of an integrable solution of the Riccati equation.
This solution is constructed in Section 5.
Finally, in Section 6 we complete the proof of Theorem \ref{theorem A}.

\section{Smooth time change of a thermostat}
Let $\gamma(t)$ be a unit speed solution of (\ref{thermo}), and $U\in C^{\infty}({\mathbb T}^2,\mathbb R)$. It is well known that we can make the following time change in the flow $\phi$
$$
s(t)=\int_0^t e^{-U(\gamma(\tau))}\,d\tau,
$$
so that the curve $\gamma_1(s):=\gamma(t(s))$ is a unit speed solution of the thermostat determined by the triple $(e^{-2U}g, e^U f, e^{2U}({\bf e}+\nabla U))$ (see \cite[Section 2.1]{P}).

We choose $U$ so that it satisfies ${\rm div}({\bf e}+\nabla U)=0$. Note that we may always find such a function. This yields $\mbox{\rm div}{\bf e_1}=0$ with respect to the metric $g_1:=e^{-2U}g$, where ${\bf e_1}:=e^{2U}({\bf e}+\nabla U)$. The flow of the thermostat $({\mathbb T}^2, g, f, {\bf e})$ is a smooth time change of the flow of the thermostat $({\mathbb T}^2,g_1, f_1, {\bf e_1})$, where $f_1:=e^Uf$.

It is easy to see that the thermostat $({\mathbb T}^2, g, f, {\bf e})$ has no conjugate points if and only if so does $({\mathbb T}^2,g_1,f_1, {\bf e_1})$. Indeed, set $\lambda_1:=f_1+\langle{\bf e}_1,iv_1\rangle_1$, where $\langle\cdot,\cdot\rangle_1$ is the inner product with respect to $g_1$ and $v_1:=e^Uv$. An easy calculation shows that
$$
d_{tv}\exp_x^{\lambda}=e^{-U}d_{s v_1}\exp_x^{\lambda_1}.
$$
Thus, $\exp_x^{\lambda}$ is a local diffeomorphism if and only if $\exp_x^{\lambda_1}$ is a local diffeomorphism. So, to prove Theorem \ref{theorem A} it is enough to show that $f=0$ and the metric $g_1$ is flat. From now on, we will consider the thermostat $({\mathbb T}^2, g_1,f_1,{\bf e}_1)$, but we will omit the subscript 1 to simplify notation.

\section{Preliminaries on thermostats}
In the next three sections $M$ denotes a closed oriented surface and $SM$ its unit sphere bundle with canonical projection $\pi:SM\to M$. The latter is in fact
a principal $S^{1}$-fibration and we let $V$ be the infinitesimal
generator of the action of $S^1$.

Given a unit vector $v\in T_{x}M$, we denote by $iv$ the
unique unit vector orthogonal to $v$ such that $\{v,iv\}$ is an
oriented basis of $T_{x}M$. There are two basic 1-forms $\alpha$
and $\beta$ on $SM$ which are defined by the formulas:
\begin{equation}\label{alpha}
\alpha_{(x,v)}(\xi):=\langle d_{(x,v)}\pi(\xi),v\rangle;
\end{equation}
\begin{equation}\label{beta}
\beta_{(x,v)}(\xi):=\langle d_{(x,v)}\pi(\xi),iv\rangle.
\end{equation}
The form $\alpha$ is the canonical contact form of $SM$ whose Reeb vector
field is the geodesic vector field $X$.
The volume form $\Theta:=\alpha\wedge d\alpha$ gives rise to the Liouville measure
on $SM$.

A basic theorem in 2-dimensional Riemannian geometry asserts that
there exists a~unique 1-form $\psi$ on $SM$ (the connection form)
such that $\psi(V)=1$ and
\begin{align}
& d\alpha=\psi\wedge \beta,\label{riem1}\\ & d\beta=-\psi\wedge
\alpha,\label{riem2}\\ & d\psi=-(K\circ\pi)\,\alpha\wedge\beta,
\label{riem3}
\end{align}
where $K$ is the Gaussian curvature of $M$. In fact, the form
$\psi$ is given by
\begin{equation}\label{psi}
\psi_{(x,v)}(\xi)=\left\langle \frac{DZ}{dt}(0),iv\right\rangle,
\end{equation}
where $Z:(-\varepsilon,\varepsilon)\to SM$ is any curve with
$Z(0)=(x,v)$,  $\dot{Z}(0)=\xi$,  and $\frac{DZ}{dt}$ is the
covariant derivative of $Z$ along the curve $\pi\circ Z$.

For later use it is convenient to introduce the vector field $H$
uniquely defined by the conditions $\beta(H)=1$ and
$\alpha(H)=\psi(H)=0$. The vector fields $X,H$ and $V$ are dual to
$\alpha,\beta$ and $\psi$ and as a consequence of (\ref{riem1}--\ref{riem3}) they satisfy the commutation relations
\begin{equation}\label{comm}
[V,X]=H,\quad [V,H]=-X,\quad [X,H]=KV.
\end{equation}
Equations (\ref{riem1}--\ref{riem3}) also imply that the vector fields
$X,H$ and $V$ preserve the volume form $\Theta$ and hence
the Liouville measure on $SM$.

Let $\lambda$ be the smooth function on $SM$ given by (\ref{thermo}), and let $F$ be the generating vector field of the thermostat flow.
Then
\begin{equation}\label{F}
F=X+\lambda V.
\end{equation}
Indeed, with $\gamma(t)=\pi\circ\phi_t(x,v)$, by straightforward calculations we get from (\ref{psi}) and (\ref{alpha}--\ref{beta})
\begin{align*}
\psi(F(x,v))&=\langle D_t\dot\gamma(0),iv\rangle=\langle \lambda(\gamma,\dot\gamma)i\dot\gamma,iv\rangle=\lambda(x,v),\\
\alpha(F(x,v))&=\langle d\pi(F(x,v)),v\rangle=\langle v,v\rangle=1,\\
\beta(F(x,v))&=\langle d\pi(F(x,v)),iv\rangle=\langle v,iv\rangle=0.
\end{align*}
Hence $F=\alpha(F)X+\beta(F)H+\psi(F)V=X+\lambda V$.

From (\ref{comm}) and (\ref{F}) we obtain:
\begin{equation}\label{bracket}
[V,F]=H-qV,\quad [V,H]=-F+\lambda V,\quad
[F,H]=-\lambda F+k V
\end{equation}
with
\begin{equation}\label{K}
q=-V(\lambda),\qquad k=K-H(\lambda)+\lambda^2.
\end{equation}

\section{Thermostats without conjugate points}
The aim of this section is to prove the following
\begin{Theorem}\label{MS3}
A thermostat $(M,g,f,\mathbf e)$ has no conjugate points if and only if all solutions of the
Jacobi equation
\begin{equation}\label{Jacobi-eq}
\ddot{y}+q\dot{y}+ky=0
\end{equation}
on any unit speed thermostat geodesic vanish at most once.
\end{Theorem}

We consider a variation of the thermostat geodesic $\gamma(t)=\pi\circ\phi_t(x,v)$ for some $(x,v)\in SM$. We set this variation to be $c(s,t)=\pi(\phi_t(Z(s)))$, where
$Z$ is a curve in $SM$ with $\dot Z(0)=\xi\in T_{(x,v)}SM$. The vector field defined as $J_{\xi}(t):=\de{c}{s}\Big|_{s=0}(t)$ is called a {\it Jacobi field} along $\gamma$
(it depends on $\xi$).

\begin{Lemma}\label{cor1}Every Jacobi field $J_{\xi}$,  written down in the form
$$
J_{\xi}(t)=x(t)\dot\gamma(t)+y(t)i\dot\gamma(t),
$$
satisfies the following {\it Jacobi equations}:
\begin{align}
\dot{x}&=\lambda y,\label{Jacobi-eqx}\\
\ddot{y}&+q\dot{y}+ky=0,\label{Jacobi-eq1}
\end{align}
with $q$ and $k$ defined by (\ref{K}). In particular, if a Jacobi field $J$ is tangent to the thermostat geodesic $\gamma$ everywhere, then $J=c\dot\gamma$, where $c=const$.

\end{Lemma}
\begin{proof}
For $\xi\in T(SM)$ write
$$
d\phi_t(\xi)=x(t)F+y(t)H+z(t)V.
$$
Equivalently,
$$
\xi=x(t)d\phi_{-t}(F)+y(t)d\phi_{-t}(H)+z(t)d\phi_{-t}(V).
$$
If we differentiate the last equality with respect to $t$ we obtain:
$$
0=\dot{x}F+\dot{y}H+y[F,H]+\dot{z}V+z[F,V].
$$
Using the bracket relations (\ref{bracket}) and regrouping we obtain \eqref{Jacobi-eqx} and \eqref{Jacobi-eq1}.
\end{proof}

Let $\gamma:[0,T]\to M$ be a unit speed thermostat geodesic with endpoints $x=\gamma(0)$ and $y=\gamma(T)$. We say that $x$ and $y$ are {\it conjugate along $\gamma$} if the exponential map $\exp_x^{\lambda}$ is singular at $T\dot\gamma(0)$, i.e., the differential $d_{T\dot\gamma(0)} \exp_x^{\lambda}$ has non-maximal rank. Note that this definition does not contradict the definition of the absence of conjugate points that was in the introduction. The latter definition shows that conjugate points on any thermostat geodesic cannot be arbitrarily close to one another.

There exists a simple but very useful relation between the singular points of the exponential map and the Jacobi fields. In the following theorem we will describe it.

\begin{Theorem}\label{conj-along-geodesic}
Let $\gamma:[0,T]\to M$ be a unit speed thermostat geodesic with endpoints $x=\gamma(0)$ and $y=\gamma(T)$. Then $x$ and $y$ are conjugate along $\gamma$ if and only if there exists a non-trivial Jacobi field $J$ along $\gamma$ satisfying $J(0)=J(T)=0$.
\end{Theorem}
\begin{proof}
Putting $v=\dot\gamma(0)$, we have $\gamma(t)=\exp^{\lambda}_x(tv)$. We need the following
\begin{Lemma}\label{0,w}
If $w\in T_x M$ then $\tilde J(t)=d_{tv}\exp^{\lambda}_x(tw)$ is a Jacobi field along $\gamma$. Moreover, $\tilde J(0)=0$, $D_{t}\tilde J(0)=w$.
\end{Lemma}
Postponing the proof of the lemma,  we finish the proof of Theorem \ref{conj-along-geodesic}. Suppose there exists a nonzero vector $w\in T_x M$ such that $d_{Tv}\exp^{\lambda}_x(w)=0$. Then,  by Lemma \ref{0,w}, $J(t)=d_{tv}\exp^{\lambda}_x(tT^{-1}w)$ is a non-trivial Jacobi field satisfying $J(0)=J(T)=0$.
Conversely, assume exists a non-trivial Jacobi field $J$ along $\gamma$ such that $J(0)=J(T)=0$. Set $w=D_t J(0)$. By Lemma \ref{0,w}, the Jacobi field $\tilde J(t)=d_{tv}\exp^{\lambda}_x(tw)$ and the Jacobi field $J(t)$ have the same initial data $\tilde J(0)=J(0)=0$ and $D_t \tilde J(0)=D_t J(0)=w$,  and therefore coincide.  In consequence, $d_{T\dot\gamma(0)}\exp^{\lambda}_x(Tw)=\tilde J(T)=J(T)=0$. This means that $w\in\ker d_{T\dot\gamma(0)}\exp^{\lambda}_x$.
\end{proof}

\begin{proof}[Proof of Lemma \ref{0,w}]
Consider the variation $c(s,t)=\exp^{\lambda}_x(t(v+sw))$ of $\gamma$. Since
$$
\de{c}{s}(s,t)=\de{\exp^{\lambda}_x(t(v+sw))}{s}=d_{t(v+sw)}\exp^{\lambda}_x(tw),
$$
the vector field $J(t)$ is a Jacobi field. The map $d_0\exp^{\lambda}_x$ is the identity map; therefore,
$$
J(0)=d_{0}\exp^{\lambda}_x(0)=0.
$$
It is well known that $D_s Y(s,t)=D_t J(s,t)$, where $Y(s,t)=\de{c}{t}(s,t)$ and $J(s,t)=\de{c}{s}(s,t)$. Hence
$$
D_t J(s,0)=D_sY(s,0)=D_s(d_0\exp^{\lambda}_x(v+sw))=\de{v+sw}{s}=w.
$$
\end{proof}

Let $\xi\in T_{(x,v)}TM$, and $Z:(-\varepsilon,\varepsilon)\to TM$ be any curve with $Z(0)=(x,v)$ and $\dot{Z}(0)=\xi$. Write $z(t)=\pi\circ Z(t)$ and define the {\it connection map}
$$
{\mathcal K}_{x,v}(\xi):=\nabla_{z}\dot{z}(0)\in T_x M.
$$
For $(x,v)\in TM$, define the {\it vertical} and {\it horizontal} subbundles by
$$
{\mathcal V}(x,v):=\ker d_{(x,v)}\pi\quad\text{and}\quad\mathcal H(x,v):=\ker\mathcal K_{(x,v)}
$$
respectively. So we obtain the following isomorphism:
$$
T_{(x,v)}TM\to T_x M\oplus T_x M,\quad
\xi\mapsto(d\pi_{(x,v)}(\xi),\mathcal K_{(x,v)}(\xi)).
$$

Define
$$
E(x,v):={\mathcal V}(x,v)\oplus\mathbb{R}F(x,v).
$$

\begin{Lemma}\label{E-cap-V=0} If  $\gamma:[0,T]\to M$ is a thermostat geodesic, then
$$
d_{\dot{\gamma}(0)}\phi_t(E)\cap {\mathcal V}(\dot{\gamma}(t))=\{0\}
$$
for every $t\in(0,T]$.
\end{Lemma}

\begin{proof}
Take $(x,v)\in SM$ and $t\in(0,T]$. From the definition of $\exp^{\lambda}$ it is straightforward that
$$\text{image}(d_{tv}\exp^{\lambda}_x)=d_{\dot{\gamma}(t)}\pi(d_{\dot{\gamma}(0)}\phi_t(E)).$$
By the absence of conjugate points, $d_{w}\exp^{\lambda}_x$ is a linear isomorphism for every $w\in T_x M$ at which $\exp^{\lambda}_x$ is defined, and the lemma follows.
\end{proof}

\begin{proof}[Proof of Theorem \ref{MS3}]
Assume that a thermostat has no conjugate points and let $\gamma(t)$,  $0\le t\le T$,  be a unit speed thermostat geodesic. Using Lemma \ref{E-cap-V=0}, we see that $d_{\dot{\gamma}(0)}\phi_t(E)$ as a graph over the horizontal subspace for $t\in(0,T]$. We can express
$$
d_{\dot{\gamma}(0)}\phi_t(E)=\graph S:=\{(v,S(t)v), v\in \mathcal H(\dot\gamma(t))\}
$$
with $S(t):\mathcal H(\dot\gamma(t))\to \mathcal V(\dot\gamma(t))$ for $t\in(0,T]$. It is proved in \cite[Lemma 3.1]{JP} that
\begin{equation}\label{dp=j,dj}
d\phi_t(\xi)=(J_{\xi}(t),\dot J_{\xi}(t)).
\end{equation}
Let $u(t):=\langle S(t)i\dot\gamma,i\dot\gamma\rangle$. By \eqref{dp=j,dj}, $\dot J_{\eta}=SJ_{\eta}$ for all $\eta\in T_{(x,v)}SM$,  whence we immediately deduce that $\dot y=uy$. Since $u$ is well defined for all $t\in(0,T]$,  it is easy to see that $y$ never vanishes for $t\in(0,T]$.

Conversely, suppose that,  for every unit speed thermostat geodesic, any solution of \eqref{Jacobi-eq} with $y(a)=y(b)=0$ is identically zero,  $y\equiv0$, and let $J(t)$ be a Jacobi field such that $J(a)=J(b)=0$. Using Lemma \ref{cor1},  we see that $J=c\dot\gamma$. As soon as $J(a)=J(b)=0$, we must have $c=0$, and the field $J$ must vanish identically. Thus, by Theorem \ref{conj-along-geodesic} there are no conjugate points.
\end{proof}

\section{Riccati equation}
Let $\gamma(t)$, $-\infty<t<+\infty$, be a complete unit speed thermostat geodesic.
The Jacobi equation on $\gamma$ is:
\begin{equation}\label{j-e}
\ddot{y}+q\dot{y}+ky=0.
\end{equation}
If $y(t)$ is a nowhere vanishing solution of \eqref{j-e},  then
$r(t)=\frac{\dot y(t)}{y(t)}$ is a solution of the {\it Riccati equation}
\begin{equation}\label{ric-1}
\dot{r}+r^2+qr+k=0.
\end{equation}

Let
$$
m(t):=\exp{\left(-\frac12\int q(t)\,dt\right)}.
$$
If $y(t)=m(t)z(t)$ then $z(t)$ is a solution of
the equation
\begin{equation}\label{j-e-z}
\ddot{z}+\tilde k z=0,
\end{equation}
where
\begin{equation}\label{tilde-k}
\tilde k(t)=k(t)-\frac{\dot q}2+\frac{q^2}4.
\end{equation}
Since $m(t)$ is nowhere zero,  equation \eqref{j-e} has no conjugate points if
and only if so does equation \eqref{j-e-z}.

The Riccati equation corresponding to \eqref{j-e-z} is
\begin{equation}\label{ric-2}
\dot{u}+u^2+\tilde k=0.
\end{equation}
Clearly, the solutions of \eqref{ric-1} and \eqref{ric-2} are related by
\begin{equation}\label{u-r}
r(t)=u(t)-q(t)/2.
\end{equation}

Observe that,  once $SM$ is compact, there is a constant $A\ge0$ such that
$$
|\tilde k(x,v)|=\left|k(x,v)-\frac{F(q(x,v))}2+\frac{q^2(x,v)}4\right|\le A^2
$$
for all $(x,v)\in SM$. Since $\tilde k(t)$ is the restriction of
$\tilde k(x,v)$ to $(\gamma,\dot\gamma)$, we have
$$
|\tilde k(t)|\le A^2.
$$

In \cite{H},  Hopf constructed a solution $u(t)$ of \eqref{ric-2} such that $|u(t)|\le A$
for all $t$. Considering all $\gamma$ gives a bounded function $u(x,v)$ on $SM$ whose
resctriction to any $\gamma$ is a solution of \eqref{ric-2}, and Hopf proves in \cite{H}
that this $u(x,v)$ is a measurable function on $SM$.
In view of \eqref{u-r},  taking $r(x,v)=u(x,v)-q(x,v)/2$
then yields a bounded measurable function $r(x,v)$ whose restriction to any $\gamma$
is a solution of \eqref{ric-1}. From \eqref{ric-1} we readily infer that $r(x,v)$
satisfies the following equation on $SM$:
\begin{equation}\label{riccati-1}
F(r)+r^2+qr+k=0.
\end{equation}

\section{Proof of Theorem \ref{theorem A}}
\subsection{Necessity}Recall the volume form $\Theta=\alpha\wedge d\alpha$  generating the Liouville measure on $SM$. The Lie derivative $L_F\Theta$ of
$\Theta$ along $F$ satisfies $L_F\Theta=V(\lambda)\Theta=-q\Theta$ (see \cite[Lemma 3.2]{DP}). An easy consequence of the Stokes theorem then yields
$$
\int_{S{\mathbb T}^2}F(r)\Theta=\int_{S{\mathbb T}^2}qr\Theta.
$$
Hence,  integrating (\ref{riccati-1}) we obtain
$$
\int_{S{\mathbb T}^2}qr\Theta+\int_{S{\mathbb T}^2}(r+q)r\Theta=-\int_{S{\mathbb T}^2}k\Theta
=-\int_{S{\mathbb T}^2}(K-H(\lambda)+\lambda^2)\Theta.
$$
Since the vector field $H$ preserves the Liouville measure, we have
$$
\int_{S{\mathbb T}^2}H(\lambda)\Theta=0,
$$
and by the Gauss-Bonnet theorem
$$
\int_{S{\mathbb T}^2}K\Theta=4\pi^2\chi({\mathbb T}^2)=0.
$$
So
\begin{equation}\label{integrated}
\int_{S{\mathbb T}^2}\left\{[V(\lambda)]^2-\lambda^2\right\}\Theta
=\int_{S{\mathbb T}^2}(q^2-\lambda^2)\Theta
=\int_{S{\mathbb T}^2}(r+q)^2\Theta\ge 0.
\end{equation}

Let $\theta_x(v)=\langle{\bf e}(x),v\rangle$. Then $\lambda(x,v)=f+V(\theta_x(v))$. Since $V$ preserves $\Theta$,  we have
$$
\int_{S{\mathbb T}^2}[V(\lambda)]^2\Theta=-\int_{S{\mathbb T}^2}\lambda V^2(\lambda)\Theta.
$$
So by (\ref{integrated}) we get
$$
\int_{S{\mathbb T}^2}\lambda(V^2(\lambda)+\lambda)\Theta=\int_{S{\mathbb T}^2}\{f^2+fV(\theta_x(v))\}\Theta\le 0.
$$
Once again using the fact that $V$ preserves $\Theta$,  we obtain
$$
\int_{S{\mathbb T}^2}fV(\theta_x(v))\Theta=0.
$$
This implies that $f=0$ and
$$
\int_{S{\mathbb T}^2}\left\{ [V(\lambda)]^2-\lambda^2\right\}\Theta=0.
$$

We find from \eqref{integrated} that $r=V(\lambda)$. Now, \eqref{riccati-1} yields
$$
K-H(\lambda)+\lambda^2+F(V(\lambda))=0.
$$

Using $\lambda(x,v)=\langle {\bf e}(x), iv\rangle=V(\theta_x(v))$ and $F=X+\lambda V$, we find
\begin{align*}
K-H(\lambda)+\lambda^2+F(V(\lambda))&=K-H(\lambda)+\lambda^2+X(V(\lambda))+\lambda V^2(\lambda)\\
&=K-H(\lambda)+X(V(\lambda))=0,
\end{align*}
where we have used that
\begin{equation}\label{id-1}
\lambda^2=-\lambda V^2(\lambda).
\end{equation}
Since (cf. the calculations in the proof of Lemma 5.2 in \cite{P})
\begin{equation}\label{id-2}
H(\lambda)-X(V(\lambda))=\mbox{div}{\bf e},
\end{equation}
we receive
$$
K=\mbox{div}{\bf e}.
$$
This completes the proof, because by Section 2 we could assume $\bf e$ solenoidal,
$\mbox{div}{\bf e}=0$. 

\subsection{Sufficiency}

When $K=0$, $f=0$ and $\mbox{div}{\bf e}=0$ by (\ref{id-1}) and (\ref{id-2}) we get
$$
\ddot{y}-\frac{d}{dt}(V(\lambda) y)=0,
$$
and
$$
\frac{d}{dt}(r-V(\lambda))+r(r-V(\lambda))=0.
$$
The Riccati equation has the solution $r=V(\lambda)$,
which shows that there are no conjugate points.

\end{document}